\documentclass{amsart}
\usepackage[utf8]{inputenc}
\usepackage{amsmath}
\usepackage{amssymb}
\usepackage{amsthm}
\usepackage{enumitem}
\usepackage{todonotes}
\usepackage{hyperref}
\usepackage{url}
\usepackage{xcolor}

\setlist[enumerate]{label=\rm{(\arabic*)}, ref=(\arabic*)}

\DeclareMathOperator{\St}{St}

\DeclareMathOperator{\rist}{Rist}

\DeclareMathOperator{\Sym}{Sym}

\newcommand{\NN}{\mathbb{N}}
\newcommand{\ZZ}{\mathbb{Z}}

\newcommand{\X}{\mathcal{X}}

\newtheorem{Mainthm}{Theorem}

\newtheorem{Maincor}[Mainthm]{Corollary}

\newtheorem*{NbhStabsIsom}{Theorem \ref{thm:NeighbourhoodStabilisersAreIsomorphic}}

\newtheorem{thm}{Theorem}[section]
\newtheorem{lemma}[thm]{Lemma}
\newtheorem{prop}[thm]{Proposition}
\newtheorem{cor}[thm]{Corollary}

\theoremstyle{definition}

\newtheorem{rem}[thm]{Remark}

\newtheorem{example}[thm]{Example}

\title{On the stabilisers of points in groups with micro-supported actions}
\author{Dominik Francoeur}
\thanks{This work was supported by the LABEX MILYON (ANR-10-LABX-0070) of Universit\'{e} de Lyon, within the program "Investissements d'Avenir" (ANR-11-IDEX-0007) operated by the French National Research Agency (ANR)}

\begin{document}

\begin{abstract}
Given a group $G$ of homeomorphism of a first-countable Hausdorff space $\X$, we prove that if the action of $G$ on $\X$ is minimal and has rigid stabilisers that act locally minimally, then the neighbourhood stabilisers of any two points in $\X$ are conjugated by a homeomorphism of $\X$. This allows us to study stabilisers of points in many classes of groups, such as topological full groups of Cantor minimal systems, Thompson groups, branch groups, and groups acting on trees with almost prescribed local actions.
\end{abstract}

\maketitle

\section{Introduction}\label{section:Introduction}

Given an action of a group $G$ on a set $\X$, one can associate to any $x\in \X$ the subgroup $\St_G(x)\leq G$ of elements of $G$ fixing $x$. Our goal in this note is to determine when two such subgroups $\St_G(x)$ and $\St_G(y)$ are isomorphic. Of course, if $x$ and $y$ are in the same orbit, their stabilisers are isomorphic, and in general, nothing else can be said. However, if $\X$ is a topological space, and if the action of $G$ on $\X$ is by homeomorphisms and is sufficiently rich, then we will show that any two regular points have isomorphic stabilisers. Using this, we will obtain, among other things, a complete classification, up to isomorphism, of stabilisers for the first Grigorchuk group and for the topological full group of a minimal homeomorphism of the Cantor set.

Before we make these statements more precise, let us first fix some terminology and notation. Let $\X$ be a topological space, and let $G$ be a group of homeomorphisms of $\X$. For $x\in \X$, we will denote by $\St_G^0(x)$ the \emph{neighbourhood stabiliser} of $x$, that is, the subgroup of all elements of $\St_G(x)$ whose set of fixed points contains a neighbourhood of $x$. Note that throughout the entirety of this text, by "neighbourhood" we will always mean open neighbourhood. It is readily seen that $\St_G^0(x)$ is a normal subgroup of $\St_G(x)$. A point $x\in \X$ is called \emph{regular} if $\St_G^0(x)=\St_G(x)$, and \emph{singular} otherwise.

Recall that the action of $G$ on $\X$ is said to be \emph{minimal} if for every $x\in \X$, the orbit of $x$ under the action of $G$ is dense in $\X$. We will say that the action is \emph{locally minimal} if for all $x\in \X$, there exists a neighbourhood $U\subseteq \X$ of $x$ such that for all $y\in U$, the intersection of the orbit of $y$ with $U$ is dense in $U$. Equivalently, the action of $G$ on $\X$ is locally minimal if and only if $\X$ can be decomposed as a disjoint union $\X=\bigsqcup_{i}W_i$ of clopen $G$-invariant subsets $W_i\subseteq \X$ such that the action of $G$ on each $W_i$ is minimal.

For a given open set $U\subseteq \X$, we will call the \emph{rigid stabiliser of $U$ in $G$}, denoted by $\rist_G(U)$, the subgroup
\[\rist_G(U) = \left\{g\in G \mid g(x)=x, \quad \forall x\in \X\setminus U\right\}\]
of all elements of $G$ that fix the complement of $U$ pointwise. A group for which $\rist_G(U)$ has a non-trivial action on $U$ for every open set $U\subseteq \X$ is said to have a \emph{micro-supported} action.

Our main result is the following:
\begin{Mainthm}\label{thm:NeighbourhoodStabilisersAreIsomorphic}
Let $\X$ be a first-countable Hausdorff space and $G$ be a group of homeomorphisms of $\X$. Let us further assume that the action of $G$ on $\X$ is minimal and that for any open set $U$, the action of $\rist_G(U)$ on $U$ is locally minimal. Then, for all $x,y\in \X$, there exists a homeomorphism $f\colon \X \rightarrow \X$ in the Ellis semigroup of $G$ such that $f\St_G^{0}(x)f^{-1}=\St_G^{0}(y)$. In particular, $\St_G^{0}(x)$ and $\St_G^{0}(y)$ are isomorphic.
\end{Mainthm}

Recall that the Ellis semigroup of a group $G$ acting on a topological space $\X$ is the closure of the image of $G$ in the set $\X^{\X}$ of all maps from $\X$ to $\X$, with the topology of pointwise convergence (see \cite{Ellis60}).

As a direct corollary of Theorem \ref{thm:NeighbourhoodStabilisersAreIsomorphic}, we get that the stabilisers of regular points for such an action are isomorphic.

\begin{Maincor}\label{cor:StabilisersRegularPointsIsomorphic}
Let $G$ and $\X$ be as in Theorem \ref{thm:NeighbourhoodStabilisersAreIsomorphic}. Then, for all regular points $x,y\in \X$, we have $\St_G(x)\cong \St_G(y)$.
\end{Maincor}


In the case where the space $\X$ is compact, we also show, using a theorem of Rubin \cite{Rubin96}, that two isomorphic stabilisers must have isomorphic groups of germs (Corollary \ref{cor:IsomorphismGerms}). In particular, in this case, the stabiliser of a regular point can never be isomorphic to the stabiliser of a singular point.


The proofs of these results are contained in Section \ref{section:Proofs}. The general strategy behind Theorem \ref{thm:NeighbourhoodStabilisersAreIsomorphic} was partly inspired by a construction of Golan and Sapir in \cite{GolanSapir17}, where they study isomorphisms between stabilisers of finite sets of points in Thompson's group $F$. However, as we work in much greater generality, our results cover a wide array of groups of interests. Indeed, in addition to Thompson's group $F$, our results also apply to Thompson's groups $V$ and $T$, to topological full groups of Cantor minimal systems, to branch groups, and to groups acting on trees with almost prescribed local actions. We briefly discuss these examples in Section \ref{section:Examples}.


\subsection*{Acknowledgements}

The author would like to thank Laurent Bartholdi, Adrien Le Boudec, Tatiana Nagnibeda, Volodymyr Nekrashevych and Aitor Pérez for useful discussions regarding this work.

\section{Isomorphisms between neighbourhood stabilisers}\label{section:Proofs}


Our main goal in this section is to prove Theorem \ref{thm:NeighbourhoodStabilisersAreIsomorphic}. We begin with a simple lemma that encapsulates the main idea behind the construction.

\begin{lemma}\label{lemma:NbhdStabilisersAreConjugate}
Let $\X$ be a topological space, let $G$ be a group of homeomorphisms of $\X$, and let $x,y\in \X$ be two points. If there exists a homeomorphism $f\colon \X \rightarrow \X$ such that
\begin{enumerate}[label=(\roman*)]
\item $f(x)=y$,\label{item:ConditionxGoesToy}
\item for any neighbourhood $U$ of $x$, there exists $g_U\in G$ such that $g_U(z)=f(z)$ for all $z\in \X\setminus U$,\label{item:ConditionSameActionOutsideNbh}
\end{enumerate}
then $f\St_G^{0}(x)f^{-1} = \St_G^{0}(y)$.
\end{lemma}
\begin{proof}



Let $g\in \St_G^{0}(x)$ be an arbitrary element. We need to show that $fgf^{-1}\in \St_G^{0}(y)$. Since $g$ is in the neighbourhood stabiliser of $x$, there must exist a neighbourhood $U$ of $x$ such that $g$ acts trivially on $U$. By condition \ref{item:ConditionSameActionOutsideNbh}, we know that there exists $g_U\in G$ such that $g_U(z)=f(z)$ for all $z\in \X\setminus U$. Therefore, we have $f^{-1}g_U(z)=z$ for all $z\in \X\setminus U$.

Since $g$ acts trivially on $U$ and $f^{-1}g_U$ acts trivially outside of $U$, they must commute, so we have
\[f^{-1}g_Ugg_U^{-1}f=g,\]
and thus $g_Ugg_U^{-1}=fgf^{-1}$. As $g_Ugg_U^{-1} \in G$, this means that $fgf^{-1}\in G$, and since $g$ is in the neighbourhood stabiliser of $x$, we must have that $fgf^{-1}$ is in the neighbourhood stabiliser of $f(x)=y$. We have thus shown that $f\St_G^{0}(x)f^{-1} \leq \St_G^{0}(y)$.

Applying the same argument to $f^{-1}$ shows that $f^{-1}\St_G^{0}(y)f \leq \St_G^{0}(x)$, and so $f\St_G^{0}(x)f^{-1} = \St_G^{0}(y)$.
\end{proof}

Thus, in order to prove that two neighbourhood stabilisers are isomorphic, it suffices to construct a homeomorphism satisfying the conditions of Lemma \ref{lemma:NbhdStabilisersAreConjugate}. We will achieve this by taking a limit of elements of our groups (in other words, an element of the Ellis semigroup). The next lemma gives us conditions under which such a limit is a homeomorphism and satisfies the hypotheses of Lemma \ref{lemma:NbhdStabilisersAreConjugate}.

\begin{lemma}\label{lemma:LimitOfHomeomorphismsIsHomeomorphismIfThingsAreNice}
Let $\X$ be a first-countable Hausdorff topological space, let $G$ be a group of homeomorphisms of $\X$, and let $x,y\in \X$ be two points of $\X$. Suppose that there exist decreasing (with respect to inclusion) bases of neighbourhoods $\{U_i\}_{i\in \NN}$ and $\{V_i\}_{i\in \NN}$ of $x$ and $y$, respectively, and a sequence $\{g_i\}_{i\in \NN}$ of elements of $G$ such that
\begin{enumerate}[label=(\roman*)]
\item $g_i(U_i) = V_i$\label{item:giUiISVi}
\item $g_{i+1}(z) = g_{i}(z)$ for all $z\in \X\setminus U_{i}$\label{item:giAgreesWithPreviousOne}
\end{enumerate}
for all $i\in \NN$. Then, the sequence $\{g_i\}_{i\in \NN}$ converges pointwise to a homeomorphism $f\colon \X \rightarrow \X$ such that $f(x)=y$. Furthermore, we have $f\St_G^{0}(x)f^{-1} = \St_G^{0}(y)$.
\end{lemma}
\begin{proof}
Let us first show that the sequence $\{g_i\}_{i\in \NN}$ converges pointwise. Notice that since $\X$ is Hausdorff, if the limit exists, it must be unique.

For all $i\in \NN$, we have $g_i(x) \in V_i$, since $x\in U_i$ and $g_i(U_i)=V_i$ by hypothesis. As $\{V_i\}_{i\in \NN}$ is a decreasing a basis of neighbourhoods of $y$, we conclude that for every neighbourhood $V$ of $y$, there exists $N\in \NN$ such that $g_i(x)\in V$ for all $i\geq N$. Thus, we have $\lim_{i\to \infty} g_i(x) = y$.

Now, for $z\in \X$ with $z\ne x$, since $\X$ is Hausdorff, we know that there must exist $N\in \NN$ such that $z\notin U_N$, and thus $z\notin U_i$ for all $i\geq N$, since the sequence is decreasing. Consequently, for every $m\in \NN$, we have by hypothesis
\[g_{N+m+1}(z) = g_{N+m}(z).\]
Thus, by induction, we get that $g_i(z) = g_N(z)$ for all $i\geq N$. Therefore, we have $\lim_{i\to \infty} g_i(z) = g_N(z)$.

We have just shown that the sequence $\{g_i\}_{i\in \NN}$ converges to a map $f\colon \X \rightarrow \X$ such that $f(x)=y$. It remains to show that this map is a homeomorphism.

To see this, let us first notice that the sequence $\{g_i^{-1}\}_{i\in \NN}$ satisfies conditions \ref{item:giUiISVi} an \ref{item:giAgreesWithPreviousOne} if we exchange $U_i$ and $V_i$. Indeed, we clearly have $g_i^{-1}(V_i) = U_i$. This implies that if $z\in \X\setminus V_i$, then we have $g_i^{-1}(z) \in \X \setminus U_i$. Now, by condition \ref{item:giAgreesWithPreviousOne}, we have
\[g_{i+1}(g_i^{-1}(z)) = g_i(g_i^{-1}(z)) = z.\]
It follows that $g_{i+1}^{-1}(z) = g_i^{-1}(z)$ for all $z\in \X\setminus V_i$.

Therefore, the arguments above also apply to the sequence $\{g_i^{-1}\}_{i\in \NN}$, which must then converge to a map $h\colon \X \rightarrow \X$ satisfying $h(y)=x$. We will show that $h$ is the inverse of $f$. We already have $h(f(x))=x$. Now, let $z\in \X$ be different from $x$. Then, as above, there exists $N\in \NN$ such that $z\in \X\setminus U_N$, and therefore $f(z) = g_N(z)$. Since $z\notin U_N$, we must have that $g_N(z)\notin V_N$, since $g_N(U_N)=V_N$. Therefore, by a similar argument to the one above, we must have
\[h(g_N(z)) = g_N^{-1}(g_N(z)) = z.\]
This shows that $h\circ f$ is the identity map on $\X$. By a symmetric argument, we find that $f\circ h$ is also the identity, so $h=f^{-1}$.

To prove that $f$ is a homeomorphism, we still need to prove that $f$ and $f^{-1}$ are both continuous. Let us prove that $f$ is an open map. To see this, it suffices to show that for every open set $U\subseteq \X$ and for every $z\in U$, there exists some open subset $U'\subseteq U$ containing $z$ and such that $f(U')$ is open. Now, if $z\ne x$, we can choose $U'$ such that $U'\cap U_N = \emptyset$ for some $N\in \NN$ large enough, since $\X$ is Hausdorff. In that case, by the above argument, we have $f(U')= g_N(U')$. As $g_N$ is a homeomorphism, we find that $f(U')$ is open.

In the case where $z=x$, since $\{U_i\}_{i\in \NN}$ is a basis of neighbourhoods, there exists $N\in \NN$ such that $U_N \subseteq U$. We can thus choose $U'=U_N$. The result will then follow as soon as we show that $f(U_N)=V_N$. We have seen above that $f(w)=g_N(w)$ for all $w\in \X\setminus U_N$. As $g_N(U_N)=V_N$, this means that
\[f(\X\setminus U_N) = g_N(\X\setminus U_N) = \X\setminus V_N.\]
It follows from the fact that $f$ is a bijection that $f(U_N)=V_N$.

We have thus shown that $f$ is an open map. By symmetry, $f^{-1}$ is also open, so $f$ is a homeomorphism.

Finally, the fact that $f\St_G^{0}(x)f^{-1} = \St_G^{0}(y)$ follows directly from Lemma \ref{lemma:NbhdStabilisersAreConjugate}. Indeed, the map $f\colon \X\rightarrow \X$ is a homeomorphism that clearly satisfies condition \ref{item:ConditionxGoesToy} of Lemma \ref{lemma:NbhdStabilisersAreConjugate}. To see that it also satisfies condition \ref{item:ConditionSameActionOutsideNbh} of that same lemma, let $U\subset \X$ be a neighbourhood of $x$. Then, as $\{U_i\}_{i\in \NN}$ is a basis of neighbourhoods of $x$, there exists $N\in \NN$ such that $U_N\subseteq U$. By what we have seen above, we have $f(z)=g_N(z)$ for all $z\in \X\setminus U_N$ and so condition \ref{item:ConditionSameActionOutsideNbh} is satisfied.
\end{proof}

Before we proceed to our main theorem and its proof, we still need a last technical lemma. 


\begin{lemma}\label{lemma:ExistsConvergenceWithNonEmptyInterior}
Let $\X$ be a first-countable topological space and let $G$ be a group of homeomorphisms of $\X$. Suppose that the action of $G$ on $\X$ is minimal, and that for every open set $U\subseteq \X$, the action of $\rist_G(U)$ on $U$ is locally minimal (see Section \ref{section:Introduction} for a definition). Let $x\in \X$ be any point, let $U\subseteq \X$ be a neighbourhood of $x$ and let $W\subseteq U$ be a neighbourhood of $x$ such that the action of $\rist_G(U)$ is minimal on $W$. Then, for every $y\in W$ and for every neighbourhood $V\subseteq U$ of $x$, there exists a sequence $\{g_i\}_{i\in \NN}$ of elements of $\rist_G(U)$ such that
\begin{enumerate}[label=(\roman*)]
\item $\lim_{i\to \infty}g_i(x) = y$\label{item:ConditionLimitxIsy}
\item $\bigcap_{i=0}^{\infty}g_i(V)$ is open and contains $y$.\label{item:ConditionNonEmptyInterior}
\end{enumerate}
\end{lemma}
\begin{proof}
Let us fix an arbitrary neighbourhood $V\subseteq U$ of $x$, and let $\mathcal{A}\subseteq W$ be the set of all elements of $W$ for which there exists a sequence of elements of $\rist_G(U)$ satisfying conditions \ref{item:ConditionLimitxIsy} and \ref{item:ConditionNonEmptyInterior}. We wish to show that $\mathcal{A}=W$.

Let us first remark that $\mathcal{A}$ is non-empty. Indeed, we clearly have $x\in \mathcal{A}$.

We will now show that $\mathcal{A}$ is open. Let $y\in \mathcal{A}$ be any point. Then, by definition, there exists a sequence $\{g_i\}_{i\in \NN}$ of elements of $\rist_G(U)$ satisfying conditions \ref{item:ConditionLimitxIsy} and \ref{item:ConditionNonEmptyInterior}. In particular, the set $V'=\bigcap_{i=0}^{\infty}g_i(V)$ is a neighbourhood of $y$. By our assumptions on the action of $G$ on $\X$, there exists a neighbourhood $W'$ of $y$ contained in $V'$ such that the action of $\rist_G(V')$ on $W'$ is minimal. We claim that $W'\subseteq \mathcal{A}$.

To see this, let $y'\in W'$ be an arbitrary element of $W'$. As $\lim_{i\to\infty}g_i(x)=y$, we may assume without loss of generality, removing a few elements from the sequence if necessary, that $g_i(x)\in W'$ for all $i\in \NN$. Since the action of $\rist_G(V')$ on $W'$ is minimal and since $\X$ is first countable, there exists a sequence $\{h_i\}_{i\in \NN}$ of elements of $\rist_G(V')$ such that $\lim_{i\to\infty}h_ig_i(x) = y'$. Notice that since $V'\subseteq U$, we have $h_ig_i\in \rist_G(U)$. Thus, to show that $y'\in \mathcal{A}$, it remains only to show that $\bigcap_{i=0}^{\infty}h_ig_i(V)$ is a neighbourhood of $y'$.

Since $h_i\in \rist_G(V')$, we have $h_i(V')=V'$ for all $i\in \NN$. It follows that $V'=\bigcap_{i=0}^{\infty}h_ig_i(V)$. Indeed, since $V'=\bigcap_{i=0}^{\infty}g_i(V)$, we have $V'\subseteq g_i(V)$, and so
\[V'=h_i(V') \subseteq h_ig_i(V)\]
for all $i\in \NN$. Therefore, $V'\subseteq \bigcap_{i=0}^{\infty}h_ig_i(V)$. To show the other inclusion, it suffices to notice that if $x'\in \left(\bigcap_{i=0}^{\infty}h_ig_i(V)\right) \setminus V'$, then since $h_i$ acts trivially outside of $V'$, we have $x'\in \left(\bigcap_{i=0}^{\infty}g_i(V) \right)\setminus V'=\emptyset$, which is absurd. We have thus shown that $\bigcap_{i=0}^{\infty}h_ig_i(V) = V'$, which is a neighbourhood of $y'$, since $y'\in W'\subseteq V'$. This shows that $y'\in \mathcal{A}$, and as $y'$ was arbitrary, we have $W'\subseteq \mathcal{A}$. Therefore, for any $y\in \mathcal{A}$, the set $\mathcal{A}$ also contains a neighbourhood of $y$, and thus is an open set.

As we have seen above, $\mathcal{A}$ is a non-empty open set. Thus, by the minimality of the action of $\rist_G(U)$ on $W$, for every $y\in W$, there exists $g\in \rist_G(U)$ such that $g(y)\in \mathcal{A}$. By definition, there exists a sequence $\{g_i\}_{i\in \NN}$ of elements of $\rist_G(U)$ such that $\lim_{i\to\infty} g_i(x) = g(y)$ and $\bigcap_{i=0}^{\infty}g_i(V)$ is a neighbourhood of $g(y)$. As $g$ is a homeomorphism, we have $\lim_{i\to\infty}g^{-1}g_i(x) = y$. Furthermore, we have
\[\bigcap_{i=0}^{\infty}g^{-1}g_i(V) = g^{-1}\left(\bigcap_{i=0}^{\infty}g_i(V)\right),\]
so $\bigcap_{i=0}^{\infty}g^{-1}g_i(V)$ is open and contains $y$. We conclude that $y\in \mathcal{A}$. As $y$ was arbitrary, this proves that $\mathcal{A}=W$.
\end{proof}

We are now in position to prove Theorem \ref{thm:NeighbourhoodStabilisersAreIsomorphic}, which we restate for convenience.

\begin{NbhStabsIsom}
Let $\X$ be a first-countable Hausdorff space and $G$ be a group of homeomorphisms of $\X$. Let us further assume that the action of $G$ on $\X$ is minimal and that for any open set $U$, the action of $\rist_G(U)$ on $U$ is locally minimal. Then, for all $x,y\in \X$, there exists a homeomorphism $f\colon \X \rightarrow \X$ in the Ellis semigroup of $G$ such that $f\St_G^{0}(x)f^{-1}=\St_G^{0}(y)$. In particular, $\St_G^{0}(x)$ and $\St_G^{0}(y)$ are isomorphic.
\end{NbhStabsIsom}
\begin{proof}
Let $x,y\in \X$ be two arbitrary points of $\X$, and let $\{U'_i\}_{i\in \NN}$, $\{V'_i\}_{i\in \NN}$ be bases of neighbourhoods of $x$ and $y$, respectively. Let us assume that for some $n\in \NN$, we have a decreasing (with respect to inclusion) collection $\{U_i\}_{0\leq i \leq n}$ of neighbourhoods of $x$, two decreasing collections $\{V_i\}_{0\leq i \leq n}$ and $\{W_i\}_{0\leq i \leq n}$ of neighbourhoods of $y$ and a collection $\{g_i\}_{0\leq i \leq n}$ of elements of $G$ such that
\begin{enumerate}[label=(\roman*)]
\item $V_i=g_i(U_i)$\label{item:ConditionV=gU}
\item $U_i \subseteq U'_i$\label{item:ConditionUibasis}
\item $V_i\subseteq V'_i$\label{item:ConditionVibasis}
\item $W_i\subseteq V_i$ and $\rist_G(V_i)$ acts minimally on $W_i$\label{item:ConditionWInV}
\item $g_i(x) \in W_i$\label{item:ConditiongxInW}
\item $g_i(z) = g_{i-1}(z)$ for all $z\in \X\setminus U_{i-1}$,\label{item:ConditionDoNotChangeOutsideU}
\end{enumerate}
for all $0\leq i \leq n$, where we set by convention $U_{-1}=\X$ and $g_{-1}=1$. We will show that we can then construct a neighbourhood $U_{n+1}$ of $x$, two neighbourhoods $V_{n+1}, W_{n+1}$ of $y$ and an element $g_{n+1}\in G$ such that the collections $\{U_i\}_{0\leq i \leq n+1}$, $\{V_i\}_{0\leq i \leq n+1}$, $\{W_i\}_{0\leq i \leq n+1}$ and $\{g_i\}_{0\leq i \leq n+1}$ also satisfy all the conditions above.

By Lemma \ref{lemma:ExistsConvergenceWithNonEmptyInterior}, there exists a sequence $\{h_j\}_{j\in \NN}$ of elements of $\rist_G(V_n)$ such that $\lim_{j\to\infty}h_j(g_n(x)) = y$ and $\bigcap_{j=0}^{\infty}h_j(V_n\cap g_n(U'_{n+1}))$ is a neighbourhood of $y$. Let us set
\[V_{n+1} = V'_{n+1} \cap W_n \cap \bigcap_{j=0}^{\infty}h_j(V_n\cap g_n(U'_{n+1})),\]
which is open and is not empty, since it contains $y$. By our assumption on the action of $G$ on $\X$, there exists a neighbourhood $W_{n+1}$ of $y$ such that $W_{n+1}\subseteq V_{n+1}$ and the action of $\rist_G(V_{n+1})$ on $W_{n+1}$ is minimal. Since $\lim_{j\to\infty}h_j(g_n(x)) = y$, there exists $N\in \NN$ such that $h_Ng_n(x) \in W_{n+1}$. Let us set $g_{n+1}= h_Ng_n$ and $U_{n+1} = g_{n+1}^{-1}(V_{n+1})$.

We now need to verify that $U_{n+1}$, $V_{n+1}$, $W_{n+1}$ and $g_{n+1}$ satisfy all the required conditions. Conditions \ref{item:ConditionV=gU}, \ref{item:ConditionVibasis}, \ref{item:ConditionWInV} and \ref{item:ConditiongxInW} are satisfied by construction. For condition \ref{item:ConditionUibasis}, notice that we have $V_{n+1}\subseteq h_Ng_n(U'_{n+1}) = g_{n+1}(U'_{n+1})$, so $U_{n+1} \subseteq U'_{n+1}$. For condition \ref{item:ConditionDoNotChangeOutsideU}, notice that if $z\in \X\setminus U_n$, then we have $g_n(z) \in \X\setminus V_n$ by condition \ref{item:ConditionV=gU}. Therefore, as $h_N\in \rist_G(V_n)$, we have $h_N(g_n(z)) = g_n(z)$. Hence, $g_{n+1}(z) = h_N(g_n(z)) = g_n(z)$ for all $z\in \X\setminus U_n$.

We still have to verify that $U_{n+1}\subseteq U_n$, $V_{n+1}\subseteq V_n$ and $W_{n+1}\subseteq W_n$, so that the sequences $\{U_i\}_{0\leq i \leq n+1}$, $\{V_i\}_{0\leq i \leq n+1}$ and $\{W_i\}_{0\leq i \leq n+1}$ are decreasing. Notice that by construction, we have 
\[W_{n+1}\subseteq V_{n+1} \subseteq W_n \subseteq V_n,\]
so it remains only to check that $U_{n+1}\subseteq U_n$. However, it follows from condition \ref{item:ConditionDoNotChangeOutsideU} that $g_{n+1}(U_n) = V_n$, so $g_{n+1}(U_{n+1})\subseteq g_{n+1}(U_n)$, which implies that $U_{n+1}\subseteq U_n$.

Therefore, starting with $U_0=V_0=W_0=\X$ and $g_0=1$, we can construct by induction infinite collections $\{U_i\}_{i\in \NN}$, $\{V_i\}_{i\in \NN}$, $\{W_i\}_{i\in \NN}$ and $\{g_i\}_{i\in \NN}$ satisfying the above properties. By condition \ref{item:ConditionUibasis}, $\{U_i\}_{i\in \NN}$ is a decreasing basis of neighbourhood of $x$, and $\{V_i\}_{i\in \NN}$ is a decreasing basis of neighbourhood of $y$ by condition \ref{item:ConditionVibasis}. Furthermore, we have $g_i(U_i) = V_i$ for all $i\in \NN$ by condition \ref{item:ConditionV=gU} and $g_{i+1}(z)=g_i(z)$ for all $z\in \X\setminus U_i$ by condition \ref{item:ConditionDoNotChangeOutsideU}. Therefore, by Lemma \ref{lemma:LimitOfHomeomorphismsIsHomeomorphismIfThingsAreNice}, the sequence $\{g_n\}_{n\in \NN}$ converges pointwise to a homeomorphism $f\colon \X \rightarrow \X$ such that $f\St_G^{0}(x)f^{-1}=\St_G^0(y)$.
\end{proof}

\begin{rem}
Note that if the rigid stabilisers do not act locally minimally, then the conclusion of Theorem \ref{thm:NeighbourhoodStabilisersAreIsomorphic} is not true in general, as we will see in Proposition \ref{prop:NonIsomorphicStabilisers}.
\end{rem}




In the preceding theorem, we established that the neighbourhood stabilisers of any two points are isomorphic by exhibiting a homeomorphism of $\X$ that conjugates them. As we will see in the next proposition, by using a theorem of Rubin, we can show that when $\X$ is compact, any isomorphism between neighbourhood stabilisers must be of this form.

\begin{prop}\label{prop:IsomorphismsComeFromHomeos}
Let $G$ be a group of homeomorphisms of a compact Hausdorff space $\X$. Let us also suppose that the action of $G$ on $\X$ is minimal and that for every open set $U\subseteq \X$, the action of $\rist_G(U)$ on $U$ is locally minimal. Let $x_1,x_2\in \X$ be two points and let $G_1,G_2\leq G$ be two subgroups such that $\St_G^0(x_i) \leq G_i\leq \St_G(x_i)$ for $i=1,2$. Then, for every isomorphism $\varphi\colon G_1\rightarrow G_2$, there exists a homeomorphism $f\colon \X\rightarrow \X$ such that $f(x_1)=x_2$ and $\varphi(g) = f\circ g \circ f^{-1}$ for all $g\in G_1$. In particular, $\varphi(\St_G^0(x_1))=\St_G^0(x_2)$.
\end{prop}
\begin{proof}
Let us consider $\X_1 = \X\setminus \{x_1\}$ and $\X_2=\X\setminus \{x_2\}$. These are two locally compact Hausdorff spaces, and as $\X$ is compact, it can naturally be seen as their Alexandroff compactification (also known as the one-point compactification).

For $i\in \{1,2\}$, since $G_i\leq \St_G(x_i)$, the action of $G_i$ on $\X$ naturally restricts to an action on $\X_i$. Let $U\subseteq \X_i$ be any open set and let $z\in U$ be any point of $U$. As $\X$ is Hausdorff, there exist two open sets $V,V'$ of $\X$ such that $z\in V$, $x_i\in V'$ and $V\cap V' = \emptyset$. Using the fact that $\X_i$ is open in $\X$, we can assume without loss of generality that $V\subseteq U$. Then, by our assumptions, there exists a neighbourhood $W\subseteq V$ of $z$ in $\X$ such that the orbit of $z$ under the action of $\rist_{G}(V)$ is dense in $W$. As $V\cap V'=\emptyset$, we have $\rist_G(V) \leq \St_G^0(x_i)\leq G_i$, and since $V\subseteq U$, we thus have $\rist_G(V)\leq \rist_{G_i}(U)$. Therefore, the orbit of $z$ under $\rist_{G_i}(U)$ is dense in $W$. It then follows from a theorem of Rubin (Theorem 3.1 in \cite{Rubin96}) that there exists a homeomorphism $f'\colon \X_1\rightarrow \X_2$ such that $\varphi(g) = f'\circ g \circ f'^{-1}$ for all $g\in G_i$. By the functoriality of the Alexandroff compactification, we can extend $f'$ to a homeomorphism $f\colon \X\rightarrow \X$ sending $x_1$ to $x_2$. This proves the first part of the claim, and the second part then immediately follows.

\end{proof}

\begin{rem}
Note that Proposition \ref{prop:IsomorphismsComeFromHomeos} does not hold in general when $\X$ is not compact, as we will see in Example \ref{ex:ThompsonFHasWeirdIsomorphisms}.
\end{rem}

As a corollary of Proposition \ref{prop:IsomorphismsComeFromHomeos}, we obtain a simple criterion to distinguish two non-isomorphic stabilisers.

\begin{cor}\label{cor:IsomorphismGerms}
Let $G$ and $\X$ be as in Proposition \ref{prop:IsomorphismsComeFromHomeos}, and let $x,y\in \X$ be two points. If the groups of germs $\St_G(x)/\St_G^0(x)$ and $\St_G(y)/\St_G^0(y)$ of $x$ and $y$ are not isomorphic, then the stabilisers $\St_G(x)$ and $\St_G(y)$ are not isomorphic. In particular, if $x$ is regular and $y$ is singular, their stabilisers are not isomorphic.
\end{cor}
\begin{proof}
By Proposition \ref{prop:IsomorphismsComeFromHomeos}, if $\varphi\colon \St_G(x)\rightarrow \St_G(y)$ is an isomorphism, we have $\varphi(\St_G^0(x))=\St_G^0(y)$. Therefore, $\varphi$ projects to an isomorphism between $\St_G(x)/\St_G^0(x)$ and $\St_G(y)/\St_G^0(y)$.
\end{proof}

\section{Examples}\label{section:Examples}

In this section, we derive a few corollaries of Theorem \ref{thm:NeighbourhoodStabilisersAreIsomorphic} and Corollary \ref{cor:IsomorphismGerms}, which are simply direct applications of this result to various groups or classes of groups of interest.

\subsection{Topological full groups}
We begin by highlighting the fact that the results apply to topological full groups. Given a group $G$ of homeomorphisms of the Cantor space $\X$, the \emph{topological full group} of $G$ is the group $[G]$ of all homeomorphisms of $\X$ that are \emph{locally in $G$}, where a homeomorphism $h$ of $\X$ is said to be locally in $G$ if for every $x\in \X$, there exist a neighbourhood $U$ of $x$ and an element $g\in G$ such that $h|_U=g|_U$. For more information about these groups, see for instance \cite{Nekrashevych19}.

\begin{cor}\label{cor:TopologicalFullGroups}
Let $G$ be a group of homeomorphisms of the Cantor space $\X$ whose action is minimal, and let $[G]$ be the topological full group of this system. Then, for all $x,y\in \X$, there exists a homeomorphism $f\colon \X\rightarrow \X$ such that $f\St_{[G]}^{0}(x)f^{-1} = \St_{[G]}^0(y)$. In particular, if $x$ and $y$ are regular points for the action of $G$ on $\X$, then $\St_{[G]}(x)$ and $\St_{[G]}(y)$ are isomorphic.
\end{cor}
\begin{proof}
Let $U\subseteq \X$ be an open set. Since the action of $G$ on $\X$ is minimal, the action of $\rist_{[G]}(U)$ on $U$ is also minimal. Indeed, for every $x\in U$ and for every open set $W\subseteq U$, there exists some $g\in G$ such that $g(x)\in W$ by minimality of the action of $G$. It follows that there exists a clopen set $V$ contained in $U$ such that $V\cap g(V) = \emptyset$ and $g(V)\subseteq W$. Thus, we have some $h\in [G]$ such that $h|_V=g|_V$, $h_{g(V)}=g^{-1}|_{g(V)}$ and $h$ acts as the identity outside of $V\cup g(V)$. Since $V\cup g(V) \subseteq U$, we have $h\in \rist_{[G]}(U)$. We have just shown that for every $x\in U$ and every open set $W\subseteq U$, there exists $h\in \rist_{[G]}(U)$ such that $g(x)\in W$. Therefore, the action of $\rist_{[G]}(U)$ on $U$ is minimal. The first part of the claim then follows directly from Theorem \ref{thm:NeighbourhoodStabilisersAreIsomorphic}.

To prove the second part, it suffices to observe that if $x$ is a regular point for the action of $G$ on $\X$, then it is also regular for the action of $[G]$, since by definition, for every $g\in [G]$, there exist $h\in G$ and a neighbourhood $U\subseteq \X$ of $x$ such that $h|_U=g|_U$.
\end{proof}

Notice that in particular, for a minimal action of $\ZZ$ on the Cantor space, the stabilisers in the topological full group of any two points are always isomorphic.

\begin{cor}
Let $(\X, \varphi)$ be a Cantor minimal system, meaning that $\varphi\colon \X \rightarrow \X$ is a homeomorphism of the Cantor space $\X$ such that the group $\langle \varphi \rangle$ acts minimally on $\X$, and let $[\varphi]$ be the topological full group of this system. Then, for all $x,y\in \X$, there exists a homeomorphism $f\colon \X\rightarrow \X$ such that $f\St_{[\varphi]}(x)f^{-1} = \St_{[\varphi]}(y)$.
\end{cor}
\begin{proof}
Since a minimal action of $\ZZ$ on an infinite space must be free, every point of $\X$ is a regular point for the action of $\langle \varphi \rangle$, and the result thus follows directly from Corollary \ref{cor:TopologicalFullGroups}.
\end{proof}

\subsection{Thompson groups}
Our results also apply to Thompson's groups $F$, $T$ and $V$, which are important examples of groups of homeomorphisms of the interval $[0,1]$, the circle and the Cantor space, respectively. For the definition and more information about these groups, see for instance \cite{CannonFloydParry96}.

\begin{cor}
Let $G$ be either Thompson's group $F$, $T$ or $V$ with its usual action on $\X$, where $\X$ is $(0,1)$, $S^1$ or the Cantor space, respectively. Then, for all $x,y\in \X$, there exists a homeomorphism $f\colon \X \rightarrow \X$ such that $f\St_G^0(x)f^{-1} = \St_G^0(y)$.
\end{cor}
\begin{proof}
In all three cases, the space $\X$ is a first countable Hausdorff space, and it is known that the action of $G$ on $\X$ is minimal and that for every open set $U\subseteq \X$, the action of $\rist_G(U)$ on $U$ is locally minimal (for $F$, this follows from Lemma 4.2 in \cite{CannonFloydParry96}, for $T$, it can be deduced easily from the fact that $F\leq T$, and for $V$, the proof is elementary). We can thus apply Theorem \ref{thm:NeighbourhoodStabilisersAreIsomorphic}.
\end{proof}

In particular, for Thompson's group $F$, if $x,y\in (0,1)$ are irrational points, then $\St_F(x)$ is isomorphic to $\St_F(y)$, since irrational points are regular points for the action of $F$. Thus, we can recover a part of the results of Golan and Sapir on the stabilisers of $F$ obtained in \cite{GolanSapir17}.

However, we cannot use Corollary \ref{cor:IsomorphismGerms} to distinguish between the stabilisers of irrational points and those of rational points, since the space $(0,1)$ is not compact. In fact, for the action of $F$ on $(0,1)$, the conclusions of Proposition \ref{prop:IsomorphismsComeFromHomeos} do not hold, as the next example shows.

\begin{example}\label{ex:ThompsonFHasWeirdIsomorphisms}
Let $\varphi\colon \St_F(\frac{1}{2}) \rightarrow \St_F(\frac{1}{2})$ be the map given by
\[\varphi(g)(x) = \begin{cases}
g(x+\frac{1}{2})-\frac{1}{2} & \text{ if } x\in[0,\frac{1}{2}]\\
g(x-\frac{1}{2})+\frac{1}{2} & \text{ if } x\in [\frac{1}{2}, 1].
\end{cases}\]
One can readily check that this map is a well-defined isomorphism. However, this isomorphism cannot be expressed as the conjugation by a homeomorphism of $(0,1)$. Indeed, if $g\in \St_F^0(\frac{1}{2})$ is an element that fixes no point in the interval $(0, d)$ for some $0<d<\frac{1}{2}$, then we see from the formula that $\varphi(g)$ fixes no point in the interval $(\frac{1}{2},\frac{1}{2}+d)$ and so does not belong to $\St_F^0(\frac{1}{2})$. In fact, one can show that $\varphi(\St_F^0(\frac{1}{2})) = \St_{[F,F]}(\frac{1}{2})$. As $\varphi(\St_F^0(\frac{1}{2}))\ne \St_F^0(\frac{1}{2})$, we conclude that $\varphi$ cannot be given by the conjugation by a homeomorphism.
\end{example}

\subsection{Branch groups}
Another class of groups to which our results apply is the class of branch groups, which are a special class of groups of automorphisms of rooted trees. We refer the interested reader to \cite{BartholdiGrigorchukSunic03} for a definition and more information about these groups.

\begin{cor}\label{cor:BranchGroups}
Let $G$ be a branch group acting on an infinite, locally finite, spherically homogeneous rooted tree $T$, and let $\partial T$ be the boundary of the tree $T$, which is homeomorphic to the Cantor set. Then, for all $x,y\in \partial T$, there exists an automorphism $f\colon T\rightarrow T$ of the rooted tree $T$ such that $f\St_G^0(x)f^{-1} = \St_G^0(y)$.
\end{cor}
\begin{proof}
It follows directly from the definition of a branch group that the action of $G$ on $\partial T$ is minimal. Furthermore, for all open set $U\subseteq \partial T$ and for all $x\in U$, there exists a finite partition $V_1\sqcup V_2 \sqcup \dots \sqcup V_n$ of $\partial T$ into clopen subsets such that $x\in V_1\subseteq U$ and $\prod_{i=1}^n\rist_G(V_i)$ has finite index in $G$. This implies\footnote{An explicit proof of this fact for groups acting on rooted trees can be found for instance in \cite{FrancoeurThesis19}, Lemma 9.2.12, but it could be proved in greater generality.} that there exists some clopen set $W\subseteq V_1$ containing $x$ such that $\prod_{i=1}^n\rist_G(V_i)$, and therefore $\rist_G(V_1) \leq \rist_G(U)$, act minimally on $W$. Therefore, the action of $G$ on $\partial T$ satisfies the hypotheses of Theorem \ref{thm:NeighbourhoodStabilisersAreIsomorphic}, so for every $x,y\in \partial T$ there exists a homeomorphism $f\colon \partial T \rightarrow \partial T$ such that $f\St_G^{0}(x)f^{-1}=\St_G^{0}(y)$.

To conclude, it remains only to see that the homeomorphism $f$ is in fact induced by an automorphism of the rooted tree $T$. However, by Theorem \ref{thm:NeighbourhoodStabilisersAreIsomorphic}, $f$ belongs to the Ellis semigroup of $G$, and since the action of $G$ on $\partial T$ is by isometries, $f$ must also be an isometry. It then follows from the definition of the metric on $\partial T$ that $f$ comes from an automorphism of $T$.
\end{proof}

Thanks to Corollaries \ref{cor:BranchGroups} and \ref{cor:IsomorphismGerms}, we can obtain a complete classification, up to isomorphism, of stabilisers of points in many branch group. As an illustration, we give below such a classification for the first Grigorchuk group.

\begin{example}
Let $X=\{0,1\}$ be a set of two elements and $X^{\NN}$ be the set of all right-infinite words on $X$. Let $\mathcal{G}$ be the first Grigorchuk group (see \cite{Grigorchuk83}, or \cite{BartholdiGrigorchukSunic03} Section 1.6.1) with its standard action on $X^{\NN}$. The set of singular points for the action of $\mathcal{G}$ on $X^{\NN}$ is the set of right-infinite words that are cofinal with the constant sequence $1^{\NN}$, which is precisely the orbit of $1^{\NN}$ under $\mathcal{G}$. Therefore, there are exactly two isomorphism classes for stabilisers in $\mathcal{G}$ of elements of $X^{\NN}$. More precisely, given $x,y\in X^{\NN}$, we have $\St_{\mathcal{G}}(x)\cong \St_{\mathcal{G}}(y)$ if and only if $x$ and $y$ are both cofinal with or both not cofinal with $1^{\NN}$.
\end{example}

\subsection{Groups acting on trees with almost prescribed local actions}

We conclude with a last class of examples, which shows that our assumptions in Theorem \ref{thm:NeighbourhoodStabilisersAreIsomorphic} are necessary for the conclusion to hold. This was pointed out to us by Adrien Le Boudec.

Let $\Omega$ be a finite set of cardinality $d$, and let $F\leq F'\leq \Sym(\Omega)$ be two groups of permutations of $\Omega$. Given a $d$-regular tree $T_d$ with a colouring $c_E\colon E(T_d)\rightarrow \Omega$ of its edges by $\Omega$, we denote by $G(F,F')$ the group of all automorphisms of $T_d$ such that the induced permutations on $\Omega$ belong to $F'$ for all vertices of $T_d$ and to $F$ for all but finitely many vertices. These groups were defined by Le Boudec in \cite{LeBoudec16}, where he calls them groups acting on trees with almost prescribed local action. We refer the reader to that article for more information about these groups. Note that by \cite{LeBoudec16}, Corollary 3.8, $G(F,F')$ is finitely generated if $F$ acts freely on $\Omega$.

Let us first notice that for all $F\leq F'$, the group $G(F,F')$ acts minimally on $\partial T_d$, the boundary of the tree, which is homeomorphic to the Cantor set. If the action of $F'$ on $\Omega$ is sufficiently rich, then the actions of the rigid stabilisers are locally minimal.

\begin{example}
Let $F\leq F'\leq \Sym(\Omega)$ be two groups of permutations such that $F$ acts simply transitively on $\Omega$ and $F'$ acts $2$-transitively on $\Omega$. Then, for all open set $U\subseteq \partial T_d$, the action of $\rist_{G(F,F')}(U)$ on $U$ is locally minimal. Therefore, by Theorem \ref{thm:NeighbourhoodStabilisersAreIsomorphic}, for every $\xi,\eta\in \partial T_d$, we have $\St_{G(F,F')}^{0}(\xi)\cong \St_{G(F,F')}^{0}(\eta)$.
\end{example}

However, there are some groups $F'$ for which the rigid stabilisers do not act locally minimally, so the hypotheses of Theorem \ref{thm:NeighbourhoodStabilisersAreIsomorphic} are not satisfied, and in some of those cases, the conclusion of the theorem does not hold, as the next proposition shows.

\begin{prop}\label{prop:NonIsomorphicStabilisers}
Let $F\leq F' \leq \Sym(\Omega)$ be two groups of permutations of $\Omega$, with $F$ acting simply transitively on $\Omega$. Let us further assume that there exist $a,b, c\in \Omega$ such that $\St_{F'}(a)=\St_{F'}(b)$ but $\St_{F'}(a)\ne \St_{F'}(c)$. Then, there exist $\xi, \eta \in \partial T_d$ such that $\St^{0}_{G(F,F')}(\xi) \not\cong \St^{0}_{G(F,F')}(\eta)$.
\end{prop}
\begin{proof}
Let $v_0\in V(T_d)$ be some vertex of $T_d$, let $\gamma$ be the bi-infinite geodesic passing through $v_0$ and whose edges are coloured alternatingly by $a$ and $b$, and let $\xi\in \partial T_d$ be one of its two endpoints. Since $\St_{F'}(a)=\St_{F'}(b)$, one can see that any element fixing a neighbourhood of $\xi$ must fix every vertex of the geodesic $\gamma$. This implies that $\St_{G(F,F')}^{0}(\xi)$ fixes an edge, and thus it follows from the \emph{edge-independence property} (see \cite{LeBoudec16} Definition 4.1) that there exist two non-trivial subgroups $K_1,K_2\leq \St_{G(F,F')}^{0}$ such that $\St_{G(F,F')}^{0}=K_1\times K_2$.

Now, let $\gamma'$ be the unique right-infinite geodesic starting at $v_0$ whose edges are coloured alternatingly by $a$ and $c$, and let $\eta\in \partial T_d$ be its endpoint. Let $H_1, H_2\leq \St_{G(F,F')}^{0}(\eta)$ be such that $\St_{G(F,F')}^{0}(\eta)=H_1\times H_2$. We claim that either $H_1=1$ or $H_2=1$. For $i\in \{1,2\}$, let
\[U_i = \left\{\zeta \in \partial T_d \setminus \{\eta\} \mid \exists h\in H_i \text{ such that } h(\zeta)\ne \zeta\right\}\]
be the set of elements moved by $H_i$. These sets are open in $\partial T_d\setminus \{\eta\}$. Furthermore, as both $H_1$ and $H_2$ are normal in $\St_{G(F,F')}^{0}(\eta)$, it follows from the fact that the action of $G(F,F')$ on $\partial T_d$ is micro-supported and from the double commutator lemma (for a reference, see for example \cite{Nekrashevych18b} Lemma 4.1) that $U_1\cap U_2 = \emptyset$.


To prove our claim, it suffices to show that either $U_1$ or $U_2$ is empty. Let $e_1, e_2, e_3, \dots \in E(T_d)$ be the edges of $\gamma'$, and for every $n\in \NN\setminus\{0\}$, let $v_{n-1}\in V(T_d)$ be the origin of $e_n$. The edge $e_n$ splits the boundary $\partial T_d$ in two clopen subsets $V_n^{+}$ and $V_n^{-}$, where $V_n^{+}$ is the set of all (equivalence classes of) right-infinite geodesics starting at $v_{n-1}$ whose first edge is $e_n$, and $V_n^{-}$ is its complement. We have $V_n^{-}\subset V_{n+1}^{-}$ and $\bigcup_{n\in \NN}V_{n+1}^{-}=\partial T_d \setminus \{\eta\}$.

Let us fix some $n\in \NN\setminus\{0\}$. Using the fact that $\St_{F'}(a)\ne \St_{F'}(c)$ and that the colouring of the edges of $\gamma'$ alternates between $a$ and $c$, one can show that there exists $g\in G(F,F')$ such that $G$ fixes $V_{n+1}^{+}$ but $g(e_n)\ne e_n$. In particular, we have $g\in \St_{G(F,F')}^{0}(\eta)$, so by our assumptions, there exist $h_1\in H_1$ and $h_2\in H_2$ such that $g=h_1h_2$. At least one of $h_1$ or $h_2$ must move $e_n$, since $g$ does. If $h_1(e_n)\ne e_n$, then it follows from the fact that $h_1\in \St_{G(F,F')}^{0}(\eta)$, and therefore must fix some $v_m$ for $m\geq n$, that $h_1(\zeta)\ne \zeta$ for all $\zeta \in V_n^{-}$, and thus we have $V_n^{-}\subseteq U_1$. Likewise, if $h_2(e_n)\ne e_n$, then we have $V_n^{-}\subseteq U_2$.

From the fact that $U_1\cap U_2 = \emptyset$ and that $V_n^{-}\subset V_{n+1}^{-}$, we conclude from the above argument that without loss of generality (up to renumbering $H_1$ and $H_2$), we have $V_n^{-}\subseteq U_1$ for all $n\in \NN\setminus\{0\}$. Since $\bigcup_{n\in \NN}V_{n+1}^{-}=\partial T_d \setminus \{\eta\}$, this implies that $U_2=\emptyset$. It follows that $\St_{G(F,F')}^{0}(\eta)$ cannot be written as a non-trivial direct product, and therefore cannot be isomorphic to $\St_{G(F,F')}^{0}(\xi)$.


\end{proof}

\bibliography{biblio}

\begin{thebibliography}{10}

\bibitem{BartholdiGrigorchukSunic03}
Laurent Bartholdi, Rostislav~I. Grigorchuk, and Zoran \v{S}uni\'{k}.
\newblock Branch groups.
\newblock In {\em Handbook of algebra, {V}ol. 3}, volume~3 of {\em Handb.
  Algebr.}, pages 989--1112. Elsevier/North-Holland, Amsterdam, 2003.

\bibitem{CannonFloydParry96}
James~W. Cannon, William~J. Floyd, and Walter~R. Parry.
\newblock Introductory notes on {R}ichard {T}hompson's groups.
\newblock {\em Enseign. Math. (2)}, 42(3-4):215--256, 1996.

\bibitem{Ellis60}
Robert Ellis.
\newblock A semigroup associated with a transformation group.
\newblock {\em Trans. Amer. Math. Soc.}, 94:272--281, 1960.

\bibitem{FrancoeurThesis19}
Dominik Francoeur.
\newblock {\em On maximal subgroups and other aspects of branch groups}.
\newblock PhD thesis, Universit\'{e} de Gen\`{e}ve, August 2019.
\newblock ID: unige:123493.

\bibitem{GolanSapir17}
Gili Golan and Mark Sapir.
\newblock On the stabilizers of finite sets of numbers in the {R}. {T}hompson
  group {$F$}.
\newblock {\em Algebra i Analiz}, 29(1):70--110, 2017.
\newblock Reprinted in St. Petersburg Math. J. {{\bfseries{2}}9} (2018), no. 1,
  51--79.

\bibitem{Grigorchuk83}
Rostislav~I. Grigorchuk.
\newblock On the {M}ilnor problem of group growth.
\newblock {\em Dokl. Akad. Nauk SSSR}, 271(1):30--33, 1983.

\bibitem{LeBoudec16}
Adrien Le~Boudec.
\newblock Groups acting on trees with almost prescribed local action.
\newblock {\em Comment. Math. Helv.}, 91(2):253--293, 2016.

\bibitem{Nekrashevych19}
V.~Nekrashevych.
\newblock Simple groups of dynamical origin.
\newblock {\em Ergodic Theory Dynam. Systems}, 39(3):707--732, 2019.

\bibitem{Nekrashevych18b}
Volodymyr Nekrashevych.
\newblock Finitely presented groups associated with expanding maps.
\newblock In {\em Geometric and cohomological group theory}, volume 444 of {\em
  London Math. Soc. Lecture Note Ser.}, pages 115--171. Cambridge Univ. Press,
  Cambridge, 2018.

\bibitem{Rubin96}
Matatyahu Rubin.
\newblock Locally moving groups and reconstruction problems.
\newblock In {\em Ordered groups and infinite permutation groups}, volume 354
  of {\em Math. Appl.}, pages 121--157. Kluwer Acad. Publ., Dordrecht, 1996.

\end{thebibliography}
\bibliographystyle{plain}

\end{document}